\documentclass[oneside, reqno, a4paper]{amsart} 
\usepackage{amsmath, amssymb, amsthm, amsfonts, bbm, calc, caption, geometry, graphicx, hyperref, ifpdf, microtype, psfrag, subcaption} 
\ifpdf\usepackage{pst-pdf}\else\fi
\usepackage[isbn=false, doi=false]{biblatex}
\graphicspath{{figures/}}

\let\oldautoref\autoref
\makeatletter
\renewcommand\autoref[1]{\@first@ref#1,@}
\def\@throw@dot#1.#2@{#1}
\def\@set@refname#1{
    \edef\@tmp{\getrefbykeydefault{#1}{anchor}{}}%
    \def\@refname{\@nameuse{\expandafter\@throw@dot\@tmp.@autorefname}s}%
}
\def\@first@ref#1,#2{%
  \ifx#2@\oldautoref{#1}\let\@secondref\@gobble
  \else%
    \@set@refname{#1}
    \@refname~\ref{#1}
    \let\@secondref\@second@ref
  \fi%
  \@secondref#2%
}
\def\@second@ref#1,#2{%
  \ifx#2@ and~\ref{#1}\let\@nextref\@gobble
  \else, \ref{#1}
    \let\@nextref\@next@ref
  \fi%
  \@nextref#2%
}
\def\@next@ref#1,#2{%
   \ifx#2@, and~\ref{#1}\let\@nextref\@gobble
   \else, \ref{#1}
   \fi%
   \@nextref#2%
}
\makeatother

\let\oldtheequation\theequation
\makeatletter
\def\tagform@#1{\maketag@@@{\ignorespaces#1\unskip\@@italiccorr}}
\renewcommand{\theequation}{(\oldtheequation)}
\makeatother 

\makeatletter
\def\ifdraft{\ifdim\overfullrule>\z@
  \expandafter\@firstoftwo\else\expandafter\@secondoftwo\fi}
\makeatother

\theoremstyle{plain}

\newtheorem{theorem}{Theorem}

\newtheorem{lemma}{Lemma}

\theoremstyle{definition}

\newcommand{\bN}{\mathbb N}

\newcommand{\bR}{\mathbb R}

\DeclareMathOperator{\arctanh}{arctanh}

\DeclareMathOperator{\Diff}{Diff}
\DeclareMathOperator{\dist}{dist}

\DeclareMathOperator{\Land}{Land}

\title[Approximation of Riemannian Distances]{Approximation of Riemannian Distances and Applications to Distance-Based Learning on Manifolds}
\author{Philipp Harms}
\address{Philipp Harms, Albert Ludwig University of Freiburg}
\email{philipp.harms@stochastik.uni-freiburg.de}
\author{Elodie Maignant}
\address{Elodie Maignant, Ecole Normale Sup\'erieure Paris-Saclay}
\email{emaignant@ens-paris-saclay.fr}
\author{Stefan Schlager}
\address{Stefan Schlager, Albert Ludwig University of Freiburg}
\email{stefan.schlager@anthropologie.uni-freiburg.de}
\thanks{We thank Peter W.~Michor, Xavier Pennec, and Stefan Sommer for helpful discussions and gratefully acknowledge support by FRIAS and Penn State University via the Freiburg--Penn State Collaboration Development Program.}
\begin{document}
\begin{abstract}
Several important algorithms for machine learning and data analysis use pairwise distances as input. On Riemannian manifolds these distances may be prohibitively costly to compute, in particular for large datasets. To tackle this problem, we propose a distance approximation which requires only a linear number of geodesic boundary value problems to be solved. The approximation is constructed by fitting a two-dimensional model space with constant curvature to each pair of samples. We demonstrate the usefulness of our approach in the context of shape analysis on landmarks spaces.  
\end{abstract}
\maketitle
\tableofcontents

\section{Introduction}

\subsection{Context}

Several important algorithms for machine learning and data analysis on manifolds use pairwise distances as an input. 
For example, this is the case for multi-dimensional scaling and agglomerative clustering. 
The main computational burden is typically not the algorithm per se, but the calculation of the pairwise distance matrix. 
Indeed, each individual distance might be the solution of a costly optimization problem, as e.g.\@ in the case of Riemannian shape analysis, and the size of the distance matrix grows quadratically in the number of samples. 
This can render distance-based learning prohibitively slow compared to alternative non-geometric methods.

\subsection{Relation to previous work}

This paper builds on the idea of \textcite{yang2011approximations} to use distance approximations for reducing the number of boundary value problems from quadratic down to linear in the number of samples.
The approximations in \cite{yang2011approximations} are based on the Baker--Campbell--Hausdorff formula and are therefore restricted to (quotients of) Lie groups with bi-invariant Riemannian metrics. 
In particular, they do not apply to landmark spaces with kernel metrics. 
Indeed, the numerical experiments in \cite{yang2011approximations} do not show any improvement of second over first order approximations.
This leads \textcite{yang2015diffeomorphic} to question if such improvements are possible at all. 
The present paper answers this question affirmatively. 

\subsection{Contribution}

We propose a new second order approximation of the distance function. 
The approximation is obtained by fitting a two-dimensional model space with constant curvature to each pair of samples, represented by tangent vectors to the mean. 
The reduction to two dimensions is crucial because higher-dimensional model spaces with prescribed sectional curvature and closed-form expressions for geodesic distances are not available, to the best of our knowledge. 
As expected, the second-order approximation is more accurate than the first-order one in small-distance regimes, which cover typical applications in shape analysis. 
Moreover, in contrast to second-order Taylor polynomials, our approximate distances are always non-negative and have more realistic large-distance asymptotics. 
As an application, we present a numerical implementation of our distance approximation on landmark manifolds with kernel metrics and demonstrate its performance on some simple toy data.

\subsection{Outlook}

In the special case of landmark manifolds with kernel metrics, our sectional curvature computations could be sped up using Mario's formula \cite{micheli2012sectional}.
This will allow us to test our algorithm on high-dimensional data from real-world applications. 
Pseudo-landmark data will be of particular interest because the small distances of adjacent pseudo-landmarks lead to large sectional curvatures, which highlights the importance of the curvature correction in the second-order distance approximations. 

\section{Approximations of pairwise distances}
\label{sec:approximations}

This aim of this section is to construct approximations of the Riemannian distance function, which are fast to compute and accurate for sufficiently concentrated sample points. 
As an auxiliary tool we first describe Taylor expansions of geodesic distances. 
Subsequently, we develop an alternative and better approximation using model spaces with constant curvature.  

\subsection{Taylor approximation of squared distances}
\label{sec:taylor}

Recall that the squared Riemannian distance function on any Riemannian manifold is smooth away from the cut locus, including on a neighborhood of the diagonal. 
This allows one to use Taylor approximations of squared distances between sufficiently close points. 
More specifically, for any two points $y,z$ near $x$, one may calculate the Riemannian logarithms $u=\log_x(y)$ and $v=\log_x(z)$ and approximate the squared distance $\dist(y,z)^2$ by a Taylor polynomial of the function $(u,v) \mapsto \dist(\exp_x(u),\exp_x(v))^2$.
This is made precise in the following lemma, which also provides an explicit expression of the Taylor series up to order five.

\begin{lemma}\label{lem:taylor}
Let $(M,g)$ be a Riemannian manifold with exponential map $\exp$,  distance function $\dist$, and curvature tensor $R$, let $x \in M$, and let $u,v \in T_xM$. 
Then it holds for sufficiently small $u$ and $v$ that
\begin{equation}\label{eq:taylor}
\dist\big(\exp_{x}(u),\exp_{x}(v)\big)^2
=
\|u-v\|_{g_x}^2-\frac{1}{3}R_x(u,v,v,u)+O(\|u\|+\|v\|)^6.
\end{equation}
\end{lemma}

\begin{proof}
This expansion is well-known even for higher orders; see e.g.\@ \cite{nicolaescu2011jets} or \cite{pennec2018taylor}.
\end{proof}

Note that the lowest-order term in the Taylor series \eqref{eq:taylor} is the Euclidean distance in normal coordinates.
The next term corrects for the influence of curvature and, by definition, improves the accuracy for small distances. 
However, for large distances, it may worsen the accuracy, can lead to negative signs, and has unrealistic quartic growth; cf.\@ \autoref{fig:curvature}.
These problems are even worse for higher order Taylor expansions, which involve covariant derivatives of the curvature tensor. 
To address these problems, we propose an alternative approximation in the next section.

\begin{figure}
\centering
\begin{psfrags}%
\psfrag{a0}[tc][tc]{$0$}%
\psfrag{a1}[tc][tc]{$\pi$}%
\psfrag{b0}[cr][cr]{$0$}%
\psfrag{b1}[cr][cr]{$\pi$}%
\psfrag{d}[cc][cc]{distance}%
\psfrag{kkkkkkkkk1}[Bl][cl]{$k=1$}%
\psfrag{kkkkkkkkk2}[Bl][cl]{$k=\tfrac12$}%
\psfrag{kkkkkkkkk3}[Bl][cl]{$k=0$}%
\psfrag{kkkkkkkkk4}[Bl][cl]{$k=-\tfrac12$}%
\psfrag{kkkkkkkkk5}[Bl][cl]{$k=-1$}%
\psfrag{t}[Bc][tc]{$t$}%
\includegraphics[width=\textwidth]{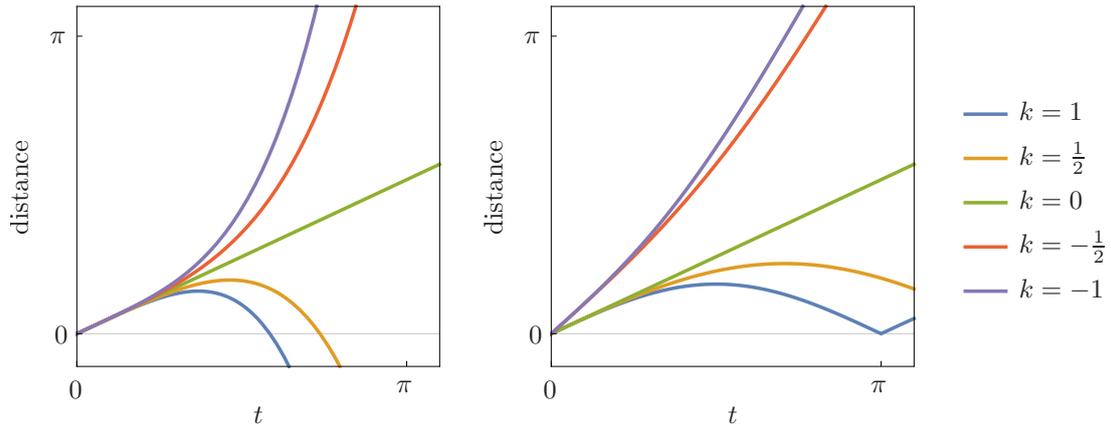}%
\end{psfrags}
\caption{Approximate distances between points $\exp_x(tu)$ and $\exp_x(tv)$ on a manifold with sectional curvature at $x$ given by $k_x(u,v) \in \{-1,-\tfrac12,0,\tfrac12,1\}$. Here $u$ and $v$ are unit-length vectors of angle $\phi=\pi/6$ in the tangent space at $x$. Left: The 2nd order Taylor approximation of the distance function (cf.\@ \autoref{lem:taylor}) has quadratic growth and may become negative. Right: The Riemannian distance on a two-dimensional model space with matching constant sectional curvature (cf.\@ \autoref{thm:curvature}) remains positive, is periodic in the case of positive curvature, and has asymptotically linear growth (cf.\@ \autoref{lem:growth}).}
\label{fig:curvature}
\end{figure}	

\subsection{Approximation by constant curvature spaces}
\label{sec:curvature}

This section corrects some shortcomings of the Taylor approximation in the previous section.
The idea is to replace the Taylor polynomial \eqref{eq:taylor} by the squared distance function of a suitably selected constant curvature space.
This guarantees non-negativity and leads to more realistic large-distance asymptotics while retaining the same order of accuracy. 
A visual comparison of the two approximations is presented in \autoref{fig:curvature}.

In more details, the proposed approximation works as follows.
For any points $y,z$ near $x \in M$, calculate the Riemannian logarithms $u=\log_x(y)$ and $v=\log_x(z)$ and sectional curvature $k_x(u,v)$, find a surface with constant curvature equal to $k_x(u,v)$, and approximate the squared distance between $y$ and $z$ by the squared distance function of the constant curvature surface.
This is made precise in the following theorem.

\begin{theorem}\label{thm:curvature}
Let $(M,g)$ be a Riemannian manifold with Riemannian distance function $\dist$ and sectional curvature $k$,
let $x \in M$, 
let $u,v\in T_xM$, 
and let $(\tilde M,\tilde g)$ be a 2-dimensional constant curvature space with $\tilde x \in \tilde M$ and $\tilde u,\tilde v \in T_{\tilde x}\tilde M$ such that
\begin{align*}
g_x(u,u)=\tilde g_{\tilde x}(\tilde u,\tilde u),
\quad
g_x(v,v)=\tilde g_{\tilde x}(\tilde v,\tilde v),
\quad
g_x(u,v)=\tilde g_{\tilde x}(\tilde u,\tilde v),
\quad
k_x(u,v)=\tilde k_{\tilde x}(\tilde u,\tilde v).
\end{align*}
Then it holds for sufficiently small $u$ and $v$ that
\begin{equation}\label{eq:curvature}
\dist\big(\exp_{x}(u),\exp_{x}(v)\big)^2
=
\widetilde\dist\big(\exp_{\tilde x}(\tilde u),\exp_{\tilde x}(\tilde v)\big)^2+O(\|u\|+\|v\|)^6.
\end{equation}
\end{theorem}

\begin{proof}
The conditions ensure that $R_x(u,v,v,u)=\tilde R_{\tilde x}(\tilde u,\tilde v,\tilde v,\tilde u)$. 
Therefore, the functions 
\begin{align*}
(u,v)&\mapsto \dist\big(\exp_{x}(u),\exp_{x}(v)\big)^2, 
& 
(\tilde u,\tilde v) &\mapsto \widetilde\dist\big(\exp_{\tilde x}(\tilde u),\exp_{\tilde x}(\tilde v)\big)^2
\end{align*}
have the same Taylor expansion up to order 5 by \autoref{lem:taylor}.
\end{proof}

The approximation in \autoref{thm:curvature} can be computed efficiently thanks to the following closed-form expression of the geodesic distance on constant curvature surfaces (i.e., the sphere, plane, and hyperbolic plane).

\begin{lemma}
\label{lem:const_curv}
Let $(M, g)$ be a 2-dimensional manifold with constant sectional curvature $k\in\bR$, 
let $r=\mathbbm 1_{\{k\neq 0\}} |k|^{-1/2}$, 
let $x \in M$, 
let $\langle\cdot,\cdot\rangle$ and $\|\cdot\|$ denote to the Riemannian metric and norm on $T_xM$, respectively, 
let $u, v \in T_xM$, 
and let $\phi \in \bR$ satisfy $\langle u,v\rangle=\|u\|\|v\|\cos\phi$. Then
\begin{align*}
\hspace{2em}&\hspace{-2em}
\dist\big(\exp_x(u),\exp_x(v)\big)
\\&=
\left\{\begin{aligned}
&r\arccos\big(\cos\|u/r\|\cos\|v/r\|+\sin\|u/r\|\sin\|v/r\|\cos\phi\big),&&k>0,\\
&\|u-v\|,&& k=0,\\
&2r \arctanh\sqrt{1-\frac{2}{1+\cosh\|u/r\|\cosh\|v/r\|-\sinh\|u/r\|\sinh\|v/r\|\cos\phi}},&&k<0, 
\end{aligned}\right.
\end{align*}
\end{lemma}

In contrast to the Taylor approximations in \autoref{lem:taylor}, the constant-curvature approximations in \autoref{thm:curvature} behave well for large distances: they have asymptotically linear growth along geodesics emanating from the same point.
This is made precise in the following lemma.

\begin{lemma}
\label{lem:growth}
The distances in \autoref{lem:const_curv} have the following (sub)-linear growth rates:
\begin{align*}
\lim_{t\to\infty} \frac{\dist\big(\exp_x(tu),\exp_x(tv)\big)}{t}
=
\left\{\begin{aligned}
& 0, && k>0, 
\\
& \|u-v\|, && k=0, 
\\
&\|u\|+\|v\|, && k<0, \cos(\phi)\neq 1,
\\
&\big|\|u\|-\|v\|\big|,&& k<0, \cos(\phi)=1.
\end{aligned}\right.
\end{align*}
\end{lemma}

\begin{proof}
As the statement is clear for spherical and Euclidean distances, as well as for points moving along the same geodesic, the only non-trivial case is $k<0$, $u\neq 0$, $v\neq 0$, and $\cos(\phi)\neq 1$. 
Let $D(t)$ be the denominator in the expression for the hyperbolic distance, i.e., 
\begin{align*}
D(t)
&= 
1+\cosh\|tu/r\|\cosh\|tv/r\|-\sinh\|tu/r\|\sinh\|tv/r\|\cos\phi
\\&=
1+\cosh\big(\|tu/r\|-\|tv/r\|\big)+(1-\cos\phi)\sinh\|tu/r\|\sinh\|tv/r\|.
\end{align*}
The latter formula shows that $D(t)$ tends to infinity for large $t$. 
Accordingly, the distance $\dist\big(\exp_x(tu),\exp_x(tv)\big)$ tends to infinity for large $t$.
This justifies the use of l'H\^opital's rule, and one obtains
\begin{align*}
\hspace{2em}&\hspace{-2em}
\lim_{t\to\infty} \frac1t \dist\big(\exp_x(tu),\exp_x(tv)\big)
=
\lim_{t\to\infty} \frac{d}{dt} 2r \arctanh\sqrt{1-\frac{2}{D(t)}} 
=
\lim_{t\to\infty} \frac{r D'(t)}{D(t)},
\end{align*}
provided that the limit on the right-hand side exists. 
One easily calculates
\begin{equation*}
rD'(t)
=
(\|u\| - \|v\| \cos\phi) \cosh\|tv/r\| \sinh\|tu/r\| + (\|v\| - \|u\| \cos\phi) \cosh\|tu/r\| \sinh\|tv/r\|.
\end{equation*}
Approximating all $\cosh$ and $\sinh$ by $\exp$, one obtains
\begin{align*}
\lim_{t\to\infty} \frac{r D'(t)}{D(t)}
&=
\lim_{t\to\infty} \frac{(\|u\| + \|v\|)(1- \cos\phi) \exp(\|tu/r\|+\|tv/r\|)}{1+\exp(\|tu/r\|+\|tv/r\|)(1-\cos\phi)}
=
\|u\| + \|v\|.
\qedhere
\end{align*}
\end{proof}

As an alternative to \autoref{thm:curvature}, it would be tempting to use $m$-dimensional instead of two-dimensional model spaces. 
This would lead to approximate distances of the same order of accuracy with the additional benefit that the triangle inequality holds. 
However, it is difficult to find model spaces with generically prescribed sectional curvatures at a point.
Moreover, the geodesic distance on such model spaces, if they exist, might not have a closed-form expression. 
This points to the advantage of fitting model spaces separately for each pair of samples, as done in \autoref{thm:curvature}.

\section{Applications to shape analysis}
\label{sec:shape}

This section demonstrates the usefulness of distance approximations in the context of shape analysis. 
More specifically, we consider shape analysis on landmark spaces with kernel metrics. 
These metrics are widely used and provide an intuitive notion of similarity. 
As the calculation of the Riemannian distance on these spaces is computationally intensive, there is a high potential for significant speed-ups via distance approximations. 

\subsection{Landmark spaces with kernel metrics}
\label{sec:landmarks}

For any $d,m\in\bN_{>0}$, landmark space $\Land^m(\bR^d)$ is the set of all configurations of $m$ distinct points in $\bR^d$. 
Landmark space is an open subset of $\bR^{d\times m}$ and therefore a manifold. 
Any kernel $k\colon\bR^d\times\bR^d\to \bR^{d\times d}$ defines a Riemannian co-metric on landmark space via
\begin{align*}
K_q(p,p) = \sum_{i,j=1}^m (p^i)^\top k(q^i,q^j)p^j\in\bR, 
\qquad
q\in \Land^m(\bR^d), p \in T^*_q \Land^m(\bR^d).
\end{align*}
The corresponding Riemannian metric is denoted by $G$. 
If $k$ is the kernel of a reproducing Hilbert space of vector fields on $\bR^d$, then $k$ can be interpreted as a right-invariant co-metric on the diffeomorphism group $\Diff(\bR^d)$ such that for any landmark configuration $\bar q \in \Land^m(\bR^d)$, the following map is a Riemannian submersion:
\begin{equation*}
(\Diff(\bR^d),k) \ni \phi \mapsto \big(\phi(\bar q^1),\dots,\phi(\bar q^m)\big) \in (\Land^m(\bR^d),K).
\end{equation*}
In this sense the co-metric on landmark space is induced by a co-metric on the diffeomorphism group on ambient space. 
A typical choice of kernel is the Gaussian kernel, which is given by 
\begin{align*}
\forall x,y\in\bR^d: \qquad k(x,y)=\exp\left(\frac{\|x-y\|^2}{2\sigma^2}\right)I^{d\times d},
\end{align*}
where $\sigma\in(0,\infty)$ and $I^{d\times d}$ denotes the $d\times d$ identity matrix.

\subsection{Computation of sectional curvatures}

The sectional curvatures on general Riemannian manifolds can be computed numerically by taking advantage of the automatic differentiation features of modern machine learning software, as demonstrated by \textcite{kuehnel2017computational}.
This is also the approach we follow in our numerics. 
Alternatively, in the specific case of landmark manifolds with kernel metrics, the sectional curvatures could also be computed directly thanks to Mario's formula, which provides closed-form expressions via the relation to the sectional curvatures on $\Diff(\bR^d)$, as described by \textcite{micheli2012sectional}.
If the kernel metric on $\Diff(\bR^d)$ was bi-invariant, which it is not, then the Riemannian exponential would coincide with the Lie group exponential, and the sectional curvatures could be computed using the Baker--Campbell--Hausdorff formula as suggested by \textcite{yang2011approximations}.

The computational complexity of the curvature computations is as follows. 
Consider a dataset of $n$ landmark configurations $\{q_1,\dots,q_n\}$ in $\Land^m(\bR^d)$, where the dimension $d$ is treated as a constant.
Then the initial registration to some fixed template configuration $\bar q$, i.e., the computation of $v_i=\log_{\bar q}(q_i)$ for each $i \in \{1,\dots,n\}$, has complexity $O(m^2n)$. 
Most importantly, only a linear number $n$ of boundary value problems has to be solved. 
Subsequently, the approximate pairwise distances can be computed at complexity $O(m^4n^2)$ from the full Riemannian curvature tensor or at complexity $O(m^2n^2)$ using Mario's formula . 
While this quadratic-in-$n$ complexity is unavoidable when fully sampled distance matrices are required, the constants in the complexity bound can be very good.
For example, if $m$ is sufficiently small to allow the computation of the full Riemannian curvature tensor at $\bar q$, then each sectional curvature can be computed by a few matrix-times-vector operations, and the corresponding approximate distances \eqref{eq:curvature} by some additional trigonometric function evaluations.

\subsection{Numerical experiments}

We tested our distance approximation on some toy datasets of landmark configurations. 
More extensive tests on larger and higher-dimensional real-world datasets are intended in future work. 

\begin{figure}%
\begin{subfigure}[t]{0.5\textwidth-1em}%
\includegraphics[width=\textwidth]{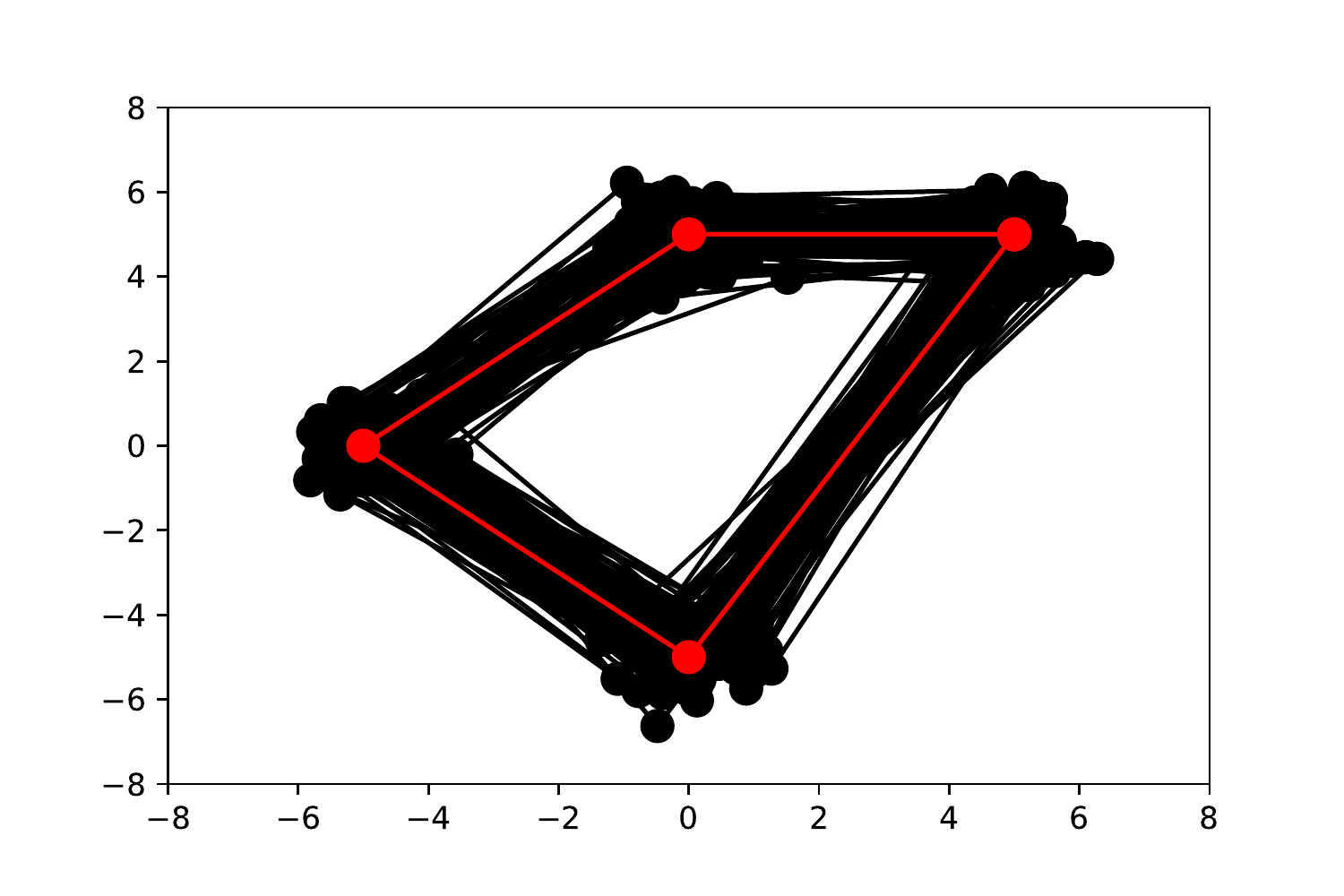}
\caption{Random shapes $q_i=\exp_{\bar q}(v_i)$, where $v_i$ are independent $N(0,G_{\bar q}/2)$-distributed random vectors.}%
\label{fig:trapezoid_a}
\end{subfigure}%
\hspace{2em}%
\begin{subfigure}[t]{0.5\textwidth-1em}%
\includegraphics[width=\textwidth]{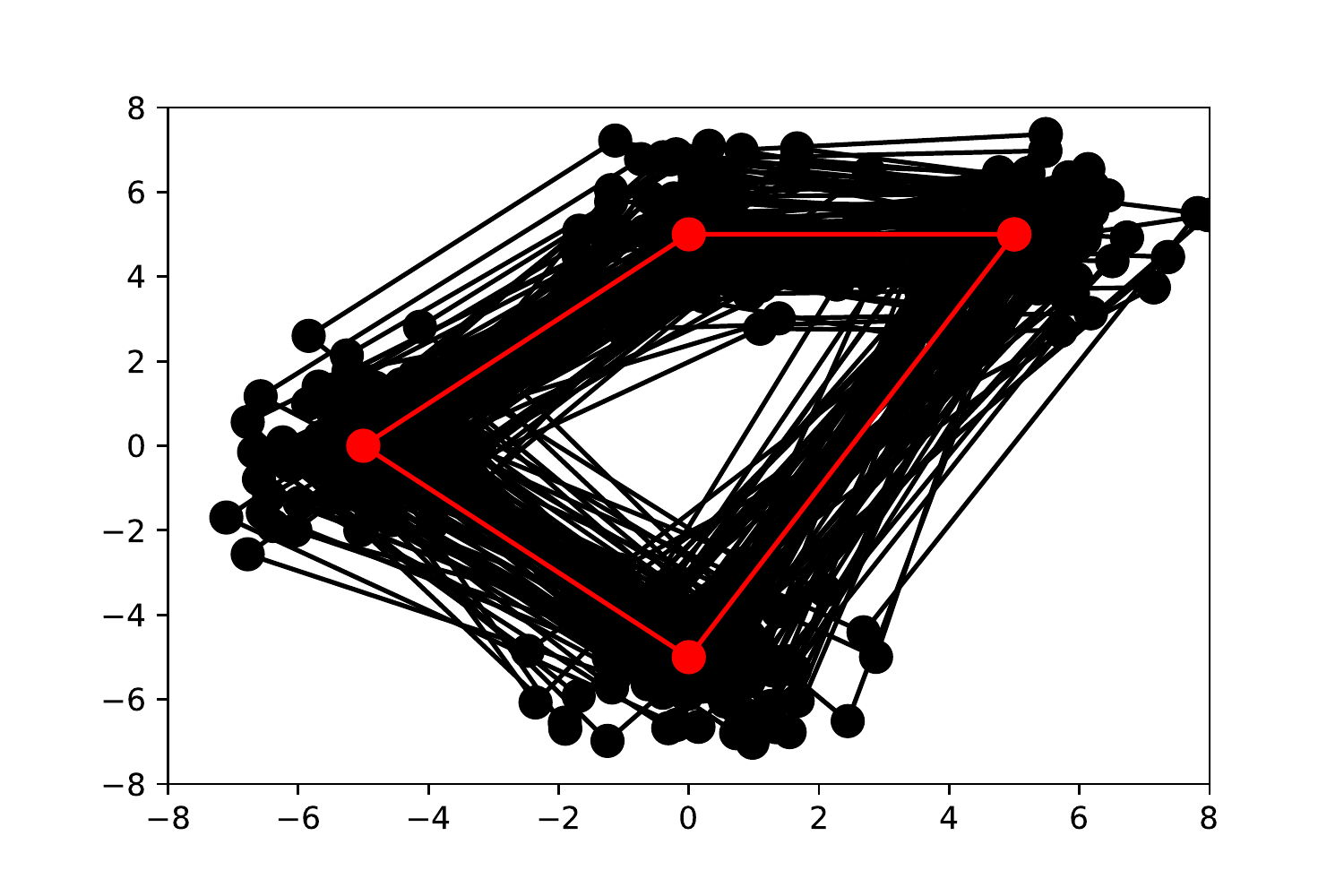}
\caption{Random shapes $q_i=\exp_{\bar q}(v_i)$, where $v_i$ are independent $N(0,G_{\bar q})$-distributed random vectors.}%
\label{fig:trapezoid_b}%
\end{subfigure}%
\\%
\begin{subfigure}[t]{0.5\textwidth-1em}%
\includegraphics[width=\textwidth]{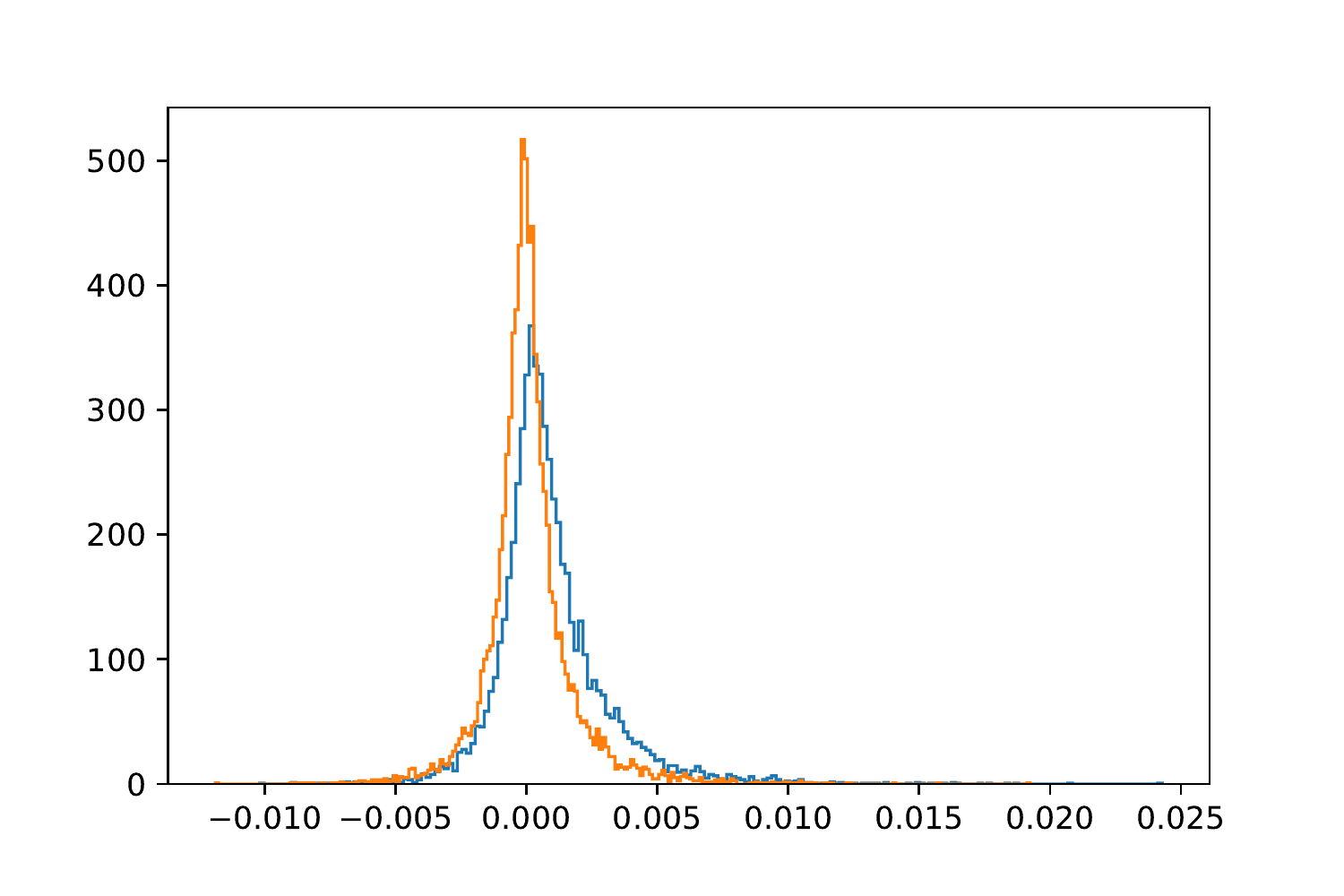}
\caption{Histogram of true minus approximate distances for the shapes in \subref{fig:trapezoid_a}.}%
\label{fig:histograms3}%
\end{subfigure}%
\hspace{2em}%
\begin{subfigure}[t]{0.5\textwidth-1em}%
\includegraphics[width=\textwidth]{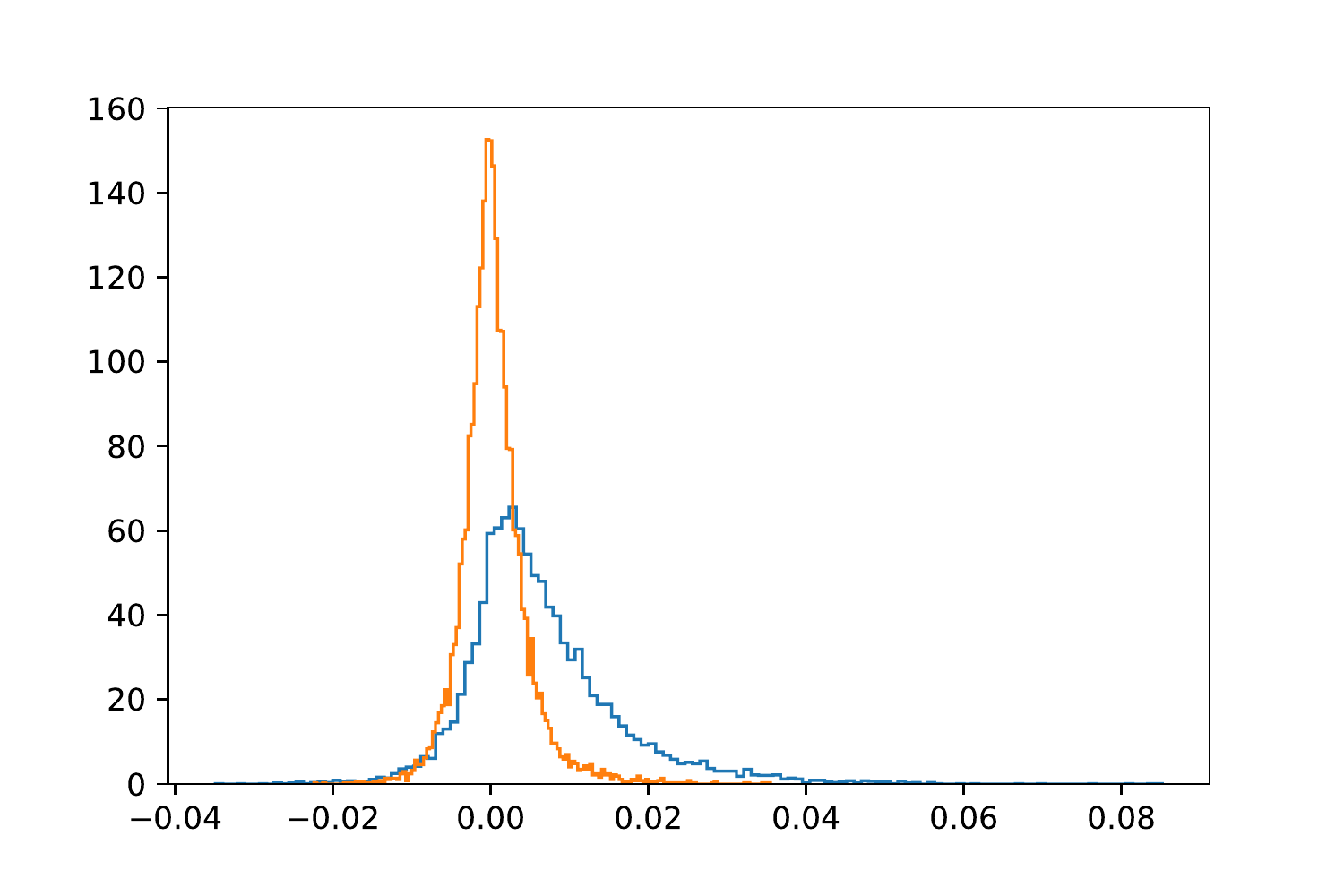}
\caption{Histogram of true minus approximate distances for the shapes in \subref{fig:trapezoid_b}.}%
\label{fig:histograms4}%
\end{subfigure}%
\\%
\begin{subfigure}[t]{0.5\textwidth-1em}%
\includegraphics[width=\textwidth]{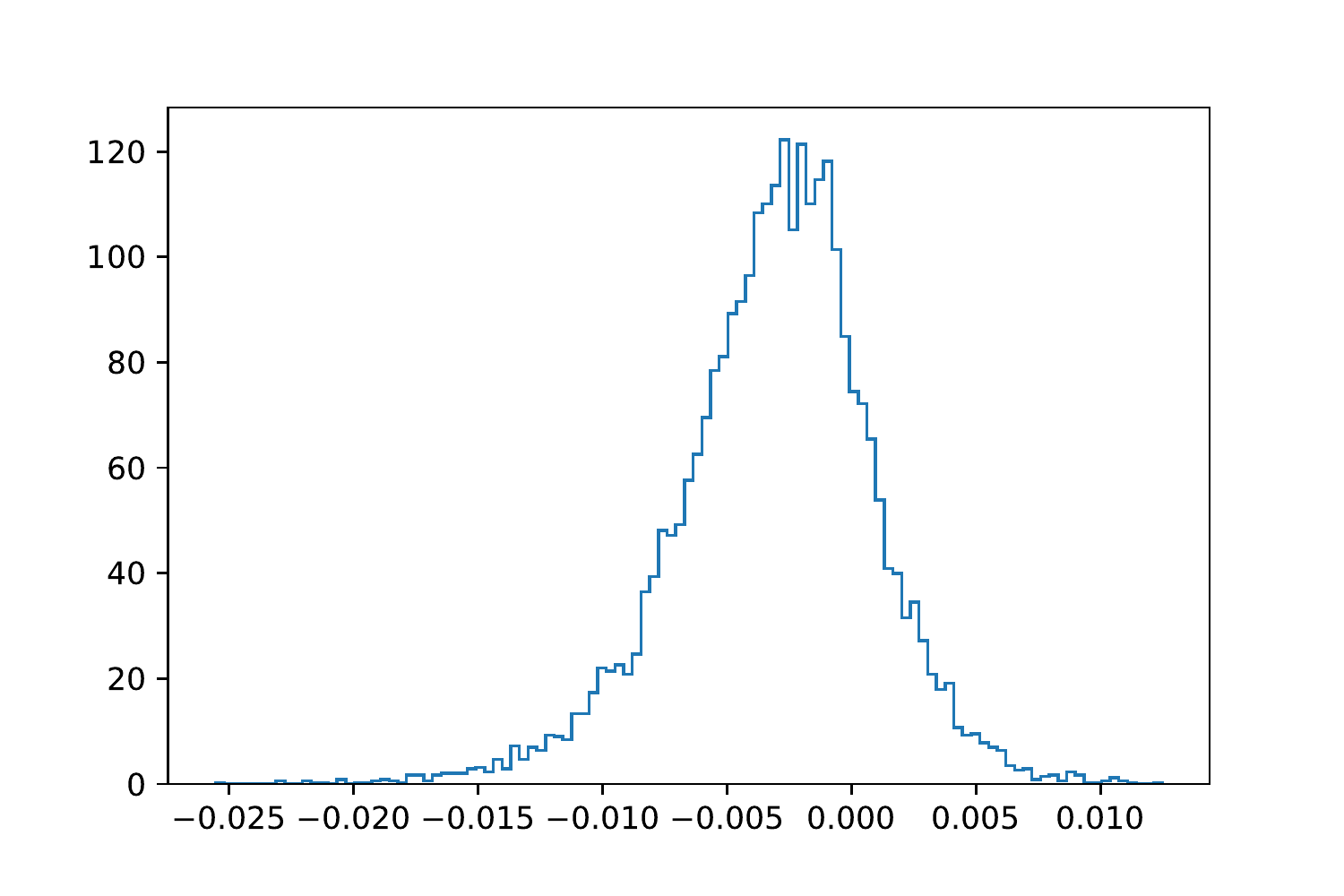}
\caption{Histogram of the sectional curvatures $k_{\bar q}(v_i,v_j)$ for the shapes in \subref{fig:trapezoid_a} and \subref{fig:trapezoid_b}.}%
\label{fig:trapezoid_e}
\end{subfigure}%
\hspace{2em}%
\begin{subfigure}[t]{0.5\textwidth-1em}%
\includegraphics[width=\textwidth]{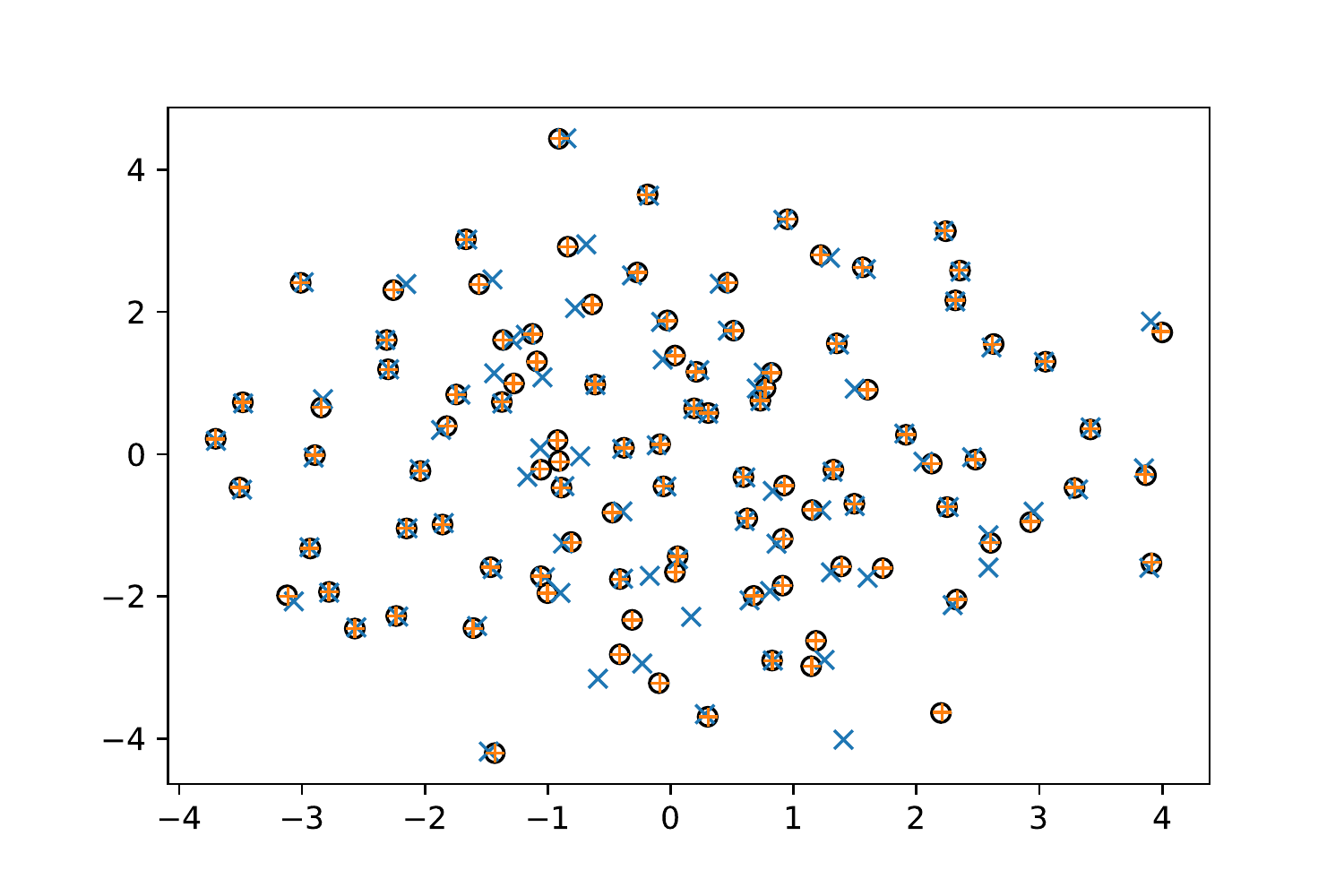}
\caption{Multi-dimensional scaling  based on true ($\circ$) and approximate ($\textcolor[HTML]{1f77b4}{\times}$, $\textcolor[HTML]{ff7f0e}{+}$) distances from \subref{fig:trapezoid_b}.}%
\label{fig:trapezoid_f}%
\end{subfigure}%
\caption{Improved accuracy of second-order (orange $\textcolor[HTML]{ff7f0e}{\blacksquare}$) over first-order (blue~$\textcolor[HTML]{1f77b4}{\blacksquare}$) distance approximations.}
\label{fig:trapezoid}
\end{figure}

A first observation of our numerical experiments is that the second-order approximation via constant-curvature spaces outperforms the first-order one, as shown in \autoref{fig:trapezoid} for standard normally distributed landmark data.
In this example the first-order approximations have a negative bias because the sectional curvatures are on average negative, but this bias is corrected by the second-order approximation. 
This confirms our theoretical predictions and stands in contrast to the results of \cite{yang2011approximations}, where an alternative second-order term did not lead to improved accuracy. 
Our findings are robust with respect to the variance of the noise and the width of the kernel: in all cases considered, the second-order approximation errors were smaller in mean and variance than the first-order ones.

A second observation is that even for moderately large datasets of approximately 15--20 shapes, the distance approximations are faster to compute than the true distances, despite the overhead of the curvature computations. 
This holds true despite the fact that our numerics do not yet rely on the explicit curvature formulas of \textcite{micheli2012sectional}, which provide potential for further improvement. 

A third observation is that the second-order correction can be highly beneficial to distance-based learning.
Indeed, \autoref{fig:trapezoid_f} shows that multi-dimensional scaling based on true distances is nearly identical to multi-dimensional scaling based on second-order approximations, whereas the result based on first-order approximations is significantly off. 

\printbibliography

\end{document}